\documentclass[12pt]{amsart}
\usepackage{amsmath}
\usepackage{amssymb}
\usepackage{amsmath,amssymb,amsthm,amscd}
\numberwithin{equation}{section}

\newtheorem{prop}{Proposition}[section]
\newtheorem{theo}{Theorem}[section]

\def\begeq{\begin{equation}}
\def\endeq{\end{equation}}

\def\lf{\left}
\def\ri{\right}

\begin{document}

\title{Rotationally symmetric harmonic diffeomorphisms between surfaces}
\author{Li Chen*, Shi-Zhong Du$^\dagger $\&   Xu-Qian Fan}
\thanks{*~Research partially supported by the National Natural Science
Foundation of China (11201131).\\
$^\dagger $Research partially supported by the National
Natural Science Foundation of China (11101106).
}
\address{Faculty of mathematics \& computer science, Hubei University, Wuhan, 430062, P. R. China.}
\email{chernli@163.com}

\address{The School of Natural Sciences and Humanities,
            Shenzhen Graduate School, The Harbin Institute of Technology, Shenzhen, 518055, P. R. China.}
\email{szdu@hitsz.edu.cn}

\address{Department of
Mathematics, Jinan University, Guangzhou,
 510632,
P. R. China.}
\email{txqfan@jnu.edu.cn}

\renewcommand{\subjclassname}{%
  \textup{2000} Mathematics Subject Classification}
\date{Feb. 2013}

\begin{abstract}
In this paper, we show that the nonexistence of rotationally symmetric harmonic diffeomorphism between the unit disk without the origin and a punctured disc with hyperbolic metric on the target.
\end{abstract}
\maketitle\markboth{Li Chen, Shi-Zhong Du$\&$ Xu-Qian Fan}{Rotationally symmetric harmonic diffeomorphisms}

\section{Introduction}
The existence of harmonic diffeomorphisms between complete Riemannian manifolds has been extensively studied, please see for example \cite{ak}-\cite{w}. In particular, Heinz \cite{hz} proved that there is no harmonic diffeomorphism from the unit disc onto $\mathbb{C}$ with its flat metric. On the other hand, Schoen \cite{rs} mentioned a question about
the existence, or nonexistence, of a harmonic diffeomorphism from the complex plane onto the
hyperbolic 2-space. At the present time, many beautiful results about the asymptotic behavior of harmonic embedding from $\mathbb{C}$ into the hyperbolic plane have been obtained, please see for example \cite{wt,httw,atw,aw}, or the review \cite{wt2} by Wan and the references therein. In 2010, Collin and Rosenberg \cite{cr} constructed harmonic diffeomorphisms from $\mathbb{C}$ onto the hyperbolic plane. In \cite{ta,ta2,rr,cl}, the authors therein studied the rotational symmetry case. One of their results is the nonexistence of rotationally symmetric harmonic diffeomorphism from $\mathbb{C}$ onto the hyperbolic plane.

In this paper, we will study the existence, or nonexistence, of rotationally symmetric harmonic diffeomorphisms from the unit disk without the origin onto a punctured disc. For simplicity, let us denote
$$\mathbb{D}^*=\mathbb{D}\setminus\{0\} \textrm{ and } P(a)=\mathbb{D}\setminus\{|z|\leq e^{-a}\} \textrm{ for } a>0,$$
here $\mathbb{D}$ is the unit disc, and $z$ is the complex coordinate of $\mathbb{C}$. We will prove the following results.
\begin{theo}\label{thm1}
For any $a>0$, there is no rotationally symmetric harmonic diffeomorphism from $\mathbb{D}^*$ onto $P(a)$ with its hyperbolic metric.
\end{theo}
And vice versa, that is:
\begin{theo}\label{thm2}
For any $a>0$, there is no rotationally symmetric harmonic diffeomorphism from $P(a)$ onto $\mathbb{D}^*$ with its hyperbolic metric.
\end{theo}
We will also consider the Euclidean case, and will prove the following theorem.
\begin{theo}\label{thm3}
For any $a>0$, there is no rotationally symmetric harmonic diffeomorphism from $\mathbb{D}^*$ onto $P(a)$ with its Euclidean metric; but on the other hand, there are rotationally symmetric harmonic diffeomorphisms from $P(a)$ onto $\mathbb{D}^*$ with its Euclidean metric.
\end{theo}

This paper is organized as follows. In section 2, we will prove
Theorem \ref{thm1} and Theorem \ref{thm2}. Theorem \ref{thm3} will be proved in section 3. At the last section, we will give an alternative proof for the nonexistence of rotationally symmetric harmonic diffeomorphism from $\mathbb{C}$ onto the hyperbolic disc.

\section*{Acknowledgments}
The author(XQ) would like to thank Prof. Luen-fai Tam for his very worthy advice.

\section{Harmonic maps from $\mathbb{D}^*$ to $P(a)$ with its hyperbolic metric and vice versa}
For convenience, let us recall the definition about the harmonic maps between surfaces. Let $M$ and $N$ be two oriented surfaces with metrics $\tau^2|dz|^2$ and $\sigma^2|du|^2$ respectively, where $z$ and $u$ are local complex coordinates of $M$ and $N$ respectively. A $C^2$ map $u$ from $M$ to $N$ is harmonic if and only if $u$ satisfies
\begin{equation} \label{eqdef}
u_{z\bar{z}}+\frac{2\sigma_u}{\sigma}u_zu_{\bar{z}}=0.
\end{equation}

Now let us prove Thoerem \ref{thm1}.
\begin{proof}[Proof of Theorem \ref{thm1}]
First of all, let us denote $(r,\theta)$ as the polar coordinates of $\mathbb{D}^*$, and $u$ as the complex coordinates of $P(a)$ in $\mathbb{C}$, then the hyperbolic metric $\sigma_1 d|u|$ on $P(a)$ can be written as
\begin{equation} \label{eqmetric1}
\frac{-\pi |du|}{a|u|\sin\lf(\frac{\pi}{a}\ln |u|\ri)}
\end{equation}
here $|u|$ is the norm of $u$ with respect to the Euclidean metric.

We will prove this theorem by contradiction. Suppose $u$ is a  rotationally symmetric harmonic diffeomorphism from $\mathbb{D}^*$ onto $P(a)$ with the metric $\sigma_1 d|u|$. Because $\mathbb{D}^*$ , $P(a)$ and the metric $\sigma_1 d|u|$ are rotationally symmetric, we can assume that such a map $u$ has the form $u=f(r)e^{i\theta}.$  Substituting $u,\ \sigma_1$ to \eqref{eqdef},  we can get
\begin{equation}\label{hdeq1}
f''+\frac{f'}{r}-\frac{f}{r^2}-\frac{\sin\lf(\frac{\pi}{a}\ln f\ri)+\frac{\pi}{a}\cos\lf(\frac{\pi}{a}\ln f\ri)}{f\sin\lf(\frac{\pi}{a}\ln f\ri)}\lf((f')^2-\frac{f^2}{r^2}\ri)=0
\end{equation}
 for  $1>r>0$. Since $u$ is a  harmonic diffeomorphism from $\mathbb{D}^*$ onto $P(a)$, we have
\begin{equation}\label{hdeq2}
 f(0)=e^{-a},\ f(1)=1 \textrm{ and } f'(r)>0 \textrm{ for  } 1>r>0,
\end{equation}
or
$$f(0)=1,\ f(1)=e^{-a} \textrm{ and } f'(r)<0 \textrm{ for  } 1>r>0.$$

We will just deal with the case that \eqref{hdeq2} is satisfied, the rest case is similar.
Let $F=\ln f\in (-a,0)$, then we have
$$F'=\frac{f'}{f}>0,\ F''=\frac{f''}{f}-\lf(\frac{f'}{f}\ri)^2.$$
Using this fact, we can get from \eqref{hdeq1} the following equation.
\begin{equation}\label{hdeqF1}
 F''+\frac{1}{r}F'-\frac{\pi}{a}\textrm{ctg}\lf(\frac{\pi}{a}F\ri)\lf(F'\ri)^2
 +\frac{1}{r^2}\frac{\pi}{a}\textrm{ctg}\lf(\frac{\pi}{a}F\ri) =0\textrm{ for  } 1>r>0
\end{equation}
with $F(0)=-a,\ F(1)=0$ and $F'(r)>0$   for  $1>r>0$.

Regarding $r$ as a function of $F$, we have the following relations.
\begin{equation}\label{hdeqFp}
F_r=r_F^{-1},\ F_{rr}=-r_F^{-3}r_{FF}.
\end{equation}
Using these facts, we can get from \eqref{hdeqF1} the following equation.
\begin{equation}\label{hdeqF2}
 \frac{r''}{r}-\lf(\frac{r'}{r}\ri)^2
 +\frac{\pi}{a}\textrm{ctg}\lf(\frac{\pi}{a}F\ri)\frac{r'}{r}
 -\lf(\frac{r'}{r}\ri)^3\frac{\pi}{a}\textrm{ctg}\lf(\frac{\pi}{a}F\ri)=0
\end{equation}
for  $0>F>-a$. Let $x=(\ln r)'(F)$, from \eqref{hdeqF2} we can get  the following equation.
$$x'+\frac{\pi}{a}\textrm{ctg}\lf(\frac{\pi}{a}F\ri)\cdot x-\frac{\pi}{a}\textrm{ctg}\lf(\frac{\pi}{a}F\ri)\cdot x^3=0.$$
One can solve this Bernoulli equation to obtain
$$x^{-2}=1+c_0\lf(\sin\lf(\frac{\pi}{a}F\ri)\ri)^2$$
here $c_0$ is a constant depending on the choice of the function $f$. So
$$x=\frac{1}{\sqrt{1+c_0\lf(\sin\lf(\frac{\pi}{a}F\ri)\ri)^2}}.$$
Since $x=(\ln r)'(F)$, we can get
\begin{equation}\label{hdeqF3}
(\ln r)(F)=\int_{0}^{F}x(t)dt
=\int_{0}^{F}\frac{1}{\sqrt{1+c_0\lf(\sin\lf(\frac{\pi}{a}t\ri)\ri)^2}}dt.
\end{equation}
Noting that $x(F)$ is continuous in $(-a,0)$ and is equal to $1$ as $F=-a$, or $0,$ one can get $x$ is uniformly bounded for $F\in [-a,0].$ So the right hand side of \eqref{hdeqF3} is uniformly bounded, but the left hand side will tend to $-\infty$ as $F\to -a.$ Hence we get a contradiction. Therefore such $f$ does not exist, Theorem \ref{thm1} has been proved.
\end{proof}

We are going to prove Theorem \ref{thm2}.
\begin{proof}[Proof of Theorem \ref{thm2}]
First of all, let us denote $(r,\theta)$ as the polar coordinates of $P(a)$, and $u$ as the complex coordinates of $\mathbb{D}^*$ in $\mathbb{C}$, then the hyperbolic metric $\sigma_2 d|u|$ on $\mathbb{D}^*$ can be written as
\begin{equation} \label{thm2-1}
\frac{|du|}{|u|\ln\frac{1}{|u|}}
\end{equation}
here $|u|$ is the norm with respect to the Euclidean metric.

We will prove this theorem by contradiction. The idea is similar to the proof of Theorem \ref{thm1}. Suppose $\psi$ is a  rotationally symmetric harmonic diffeomorphism from $P(a)$ onto $\mathbb{D}^*$ with the metric $\sigma_2 d|u|$, with the form $\psi=g(r)e^{i\theta},$  then  substituting $\psi,\ \sigma_2$ to $u,\ \sigma$ in \eqref{eqdef} respectively,  we can get
\begin{equation}\label{thm2-2}
g''+\frac{g'}{r}-\frac{g}{r^2}-\frac{1+\ln g}{g\ln g}\lf((g')^2-\frac{g^2}{r^2}\ri)=0
\end{equation}
 for  $1>r>e^{-a}$. Since $v$ is a  harmonic diffeomorphism from $P(a)$ onto $\mathbb{D}^*$, we have
\begin{equation}\label{thm2-3}
 g(e^{-a})=0,\ g(1)=1 \textrm{ and } g'(r)>0 \textrm{ for  } 1>r>e^{-a},
\end{equation}
or
$$g(e^{-a})=1,\ g(1)=0 \textrm{ and } g'(r)<0 \textrm{ for  } 1>r>e^{-a}.$$

We will only deal with the case that \eqref{thm2-3} is satisfied, the rest case is similar.
Let $G=\ln g$, then the equation \eqref{thm2-2} can be rewritten as
\begin{equation}\label{thm2-4}
G''+\frac{1}{r}G'-\frac{1}{G}(G')^2+\frac{1}{r^2G}=0
\end{equation}
for  $1>r>e^{-a}$, with $G(1)=0$ and $\lim_{r\to e^{-a}}G(r)=-\infty.$

Regarding $r$ as a function of $G$, using a similar formula of \eqref{hdeqFp}, from \eqref{thm2-4} we can get
\begin{equation}\label{thm2-5}
\frac{r''}{r}-\lf(\frac{r'}{r}\ri)^2+\frac{r'}{rG}-\frac{1}{G}\lf(\frac{r'}{r}\ri)^3=0,\ G\in(-\infty,0).
\end{equation}
Similar to solving \eqref{hdeqF2}, we can get the solution to \eqref{thm2-5} as follows.
\begin{equation}\label{thm2-6}
(\ln r)'(G)=\frac{1}{\sqrt{1+c_1G^2}},\ G\in(-\infty,0)
\end{equation}
for some nonnegative constant $c_1$ depending on the choice of $g$.\\
 If $c_1$ is equal to $0$, then $g=r$, this is  in contradiction to  \eqref{thm2-3}.\\
  If $c_1$ is positive, then taking integration on both sides of \eqref{thm2-6}, we can get
\begin{equation}\label{thm2-7}
\begin{split}
  (\ln r)(G)&=\int_{0}^{G}\frac{1}{\sqrt{1+c_1t^2}}dt\\
         &=\frac{1}{\sqrt{c_1}}\ln\lf(\sqrt{c_1}G+\sqrt{1+c_1G^2}\ri).
   \end{split}
\end{equation}
So
$$r=\lf(\sqrt{c_1}\ln g+\sqrt{1+c_1\ln^2g}\ri)^{\frac{1}{\sqrt{c_1}}},$$
with $\lim_{g\to 0+}r(g)=0.$ On the other hand, from  \eqref{thm2-3}, we have $r(0)=e^{-a}.$  Hence we get a contradiction. Therefore such $g$ does not exist, Theorem \ref{thm2} has been proved.
\end{proof}

\section{Harmonic maps from $\mathbb{D}^*$ to $P(a)$ with its Euclidean metric and vice versa}
Now let us consider the case of that the target has the Euclidean metric.
\begin{proof}[Proof of Theorem \ref{thm3}]
Let us prove the first part of this theorem, that is,  show the nonexistence of rotationally symmetric harmonic diffeomorphism from $\mathbb{D}^*$ onto $P(a)$ with its Euclidean metric. The idea is similar to the proof of Theorem \ref{thm1}, so we just sketch the proof here. Suppose there is such a harmonic diffeomorphism $\varphi$ from $\mathbb{D}^*$ onto $P(a)$ with its Euclidean metric with the form $\varphi=h(r)e^{i\theta}$, then we can get
\begin{equation}\label{thm3-1}
h''+\frac{1}{r}h'-\frac{1}{r^2}h=0 \textrm{ for  } 1>r>0
\end{equation}
with
\begin{equation}\label{thm3-2}
h(0)=e^{-a},\ h(1)=1 \textrm{ and } h'(r)>0 \textrm{ for  } 1>r>0,
\end{equation}
or
$$h(0)=1,\ h(1)=e^{-a} \textrm{ and } h'(r)<0 \textrm{ for  } 1>r>0.$$

We will just deal with the case that \eqref{thm3-2} is satisfied, the rest case is similar.
Let $H=(\ln h)'>0\textrm{ for  } 1>r>0$, then we can get
\begin{equation}\label{thm3-3}
H'+H^2+\frac{1}{r}H=\frac{1}{r^2} \textrm{ for  } 1>r>0.
\end{equation}
Solving this equation, we can get
\begin{equation}\label{thm3-4}
H=\frac{1}{r}-\frac{2}{c_3r^3+r} =-\frac{1}{r}+\frac{2c_3r}{1+c_3r^2}
\end{equation}
here $c_3$  is a constant depending on the choice of $h$ such that $H>0\textrm{ for  } 1>r>0$ which implies
\begin{equation}\label{thm3-5j}
1+c_3r^2>2 \textrm{ or } 1+c_3r^2<0 \textrm{ for } 1>r>0.
\end{equation}
So
\begin{equation}\label{thm3-5}
\ln h=-\ln r+\ln|1+c_3r^2|+c_4
\end{equation}
here $c_4$ is a constant depending on the choice of $h$. Hence
\begin{equation}\label{thm3-6}
h=|1+c_3r^2|r^{-1}e^{c_4}.
\end{equation}
From \eqref{thm3-5j} and \eqref{thm3-6}, we can get $\lim_{r\to 0}h(r)=\infty$. On the other hand, from \eqref{thm3-2},
$h(0)=e^{-a}$. We get a contradiction. Hence such a function $h$ does not exist, the first part of Theorem \ref{thm3} holds.

Now let us prove the second part of this theorem, that is, show the existence of rotationally symmetric harmonic diffeomorphisms from $P(a)$ onto $\mathbb{D}^*$ with its Euclidean metric. It suffices to find a map from $P(a)$ onto $\mathbb{D}^*$ with the form $q(r)e^{i\theta}$ such that
\begin{equation}\label{thm3-7}
q''+\frac{1}{r}q'-\frac{1}{r^2}q=0 \textrm{ for  } 1>r>e^{-a}
\end{equation}
with $q(e^{-a})=0,\ q(1)=1$ and $q'>0$ for $ 1>r>e^{-a}.$ Using the boundary condition and \eqref{thm3-6}, we can get
that
$$q=(e^{2a}r^2-1)r^{-1}(e^{2a}-1)^{-1}$$
is a solution to \eqref{thm3-7}.

Therefore we finished the proof  of Theorem \ref{thm3}.
\end{proof}

\section{Harmonic maps from $\mathbb{C}$ to the hyperbolic disc}
In this section, we will give another proof of the following result which has been proved in \cite{ta,ta2,rr,cl}.
\begin{prop}\label{p}
There is no rotationally symmetric harmonic diffeomorphism from $\mathbb{C}$ onto the hyperbolic disc.
\end{prop}
\begin{proof}
It well-known that the hyperbolic metric on the unit disc is $\frac{2}{1-|z|^2}|dz|$. We will also use the idea of the proof of Theorem \ref{thm1}. Suppose there is such a harmonic diffeomorphism $\phi$ from $\mathbb{C}$ onto $\mathbb{D}$ with its hyperbolic metric with the form $\phi=k(r)e^{i\theta}$, then  we can get
\begin{equation}\label{prop-1}
k''+\frac{1}{r}k'-\frac{1}{r^2}k+\frac{2k}{1-k^2}\lf[(k')^2-\frac{k^2}{r^2}\ri]=0 \textrm{ for  } r>0
\end{equation}
with
\begin{equation}\label{prop-2}
k(0)=0 \textrm{ and } k'(r)>0 \textrm{ for  } r>0.
\end{equation}
Regarding $r$ as a function of $k$, setting $v=(\ln r)'(k)$, the equation \eqref{prop-1} can be rewritten as
\begin{equation}\label{prop-3}
(1-k^2)v'-2kv+v^3(k+k^3)=0.
\end{equation}
That is,
\begin{equation}\label{prop-4}
(v^{-2})'+\frac{4k}{1-k^2}v^{-2}=\frac{2(k+k^3)}{1-k^2}.
\end{equation}
One can solve this equation to obtain
$$v^{-2}=k^2+c_5(1-k^2)^2$$
for some nonnegative constant $c_5$ depending on the choice of the function $k$.\\
If $c_5=0$, then we can get $r=c_6k$ for some constant $c_6$. On the other hand, $\phi$ is a diffeomorphism, so $k\to 1$ as $r\to \infty.$  This is a contradiction.\\
If $c_5>0$, then $b_1\geq k^2+c_5(1-k^2)^2\geq b_2$ for some positive constants $b_1$ and $b_2$. So
$$|(\ln r)'(k)|\leq \frac{1}{\sqrt{b_2}}.$$
This is in contradiction to the assumption that $r\to \infty$ as $k\to 1$.

Therefore Proposition  \ref{p} holds.
\end{proof}

\end{document}